\documentclass{article}[11pt]
  
\usepackage{amsfonts}
\usepackage{amsmath}
\usepackage{amssymb}
\usepackage{wasysym}
\usepackage{yfonts}  
\usepackage{amsthm}
\usepackage{amscd}   
\usepackage[all]{xypic}
\usepackage{mathrsfs}
  \usepackage{amsxtra}
\usepackage{amsbsy}
\usepackage{amstext}
\usepackage{amsopn}
\usepackage{enumerate}

\newenvironment{enumerate*}{
\begin{enumerate}[{\rm (i)}]
  \setlength{\itemsep}{3.5pt}
  \setlength{\parskip}{0pt}
}{\end{enumerate}}

\newenvironment{enumerate!}{
\begin{enumerate}[{\rm I.}]
  \setlength{\itemsep}{3.5pt}
  \setlength{\parskip}{0pt}
}{\end{enumerate}}

\newenvironment{enumeratenum}{
\begin{enumerate}[{\rm (1)}]
  \setlength{\itemsep}{3.5pt}
  \setlength{\parskip}{0pt}
}{\end{enumerate}}

\CompileMatrices

\numberwithin{equation}{section}

\newtheorem{theorem}{Theorem}[section]
\newtheorem{lemma}[theorem]{Lemma}
\newtheorem{proposition}[theorem]{Proposition}
\newtheorem{corollary}[theorem]{Corollary} 
\newtheorem*{definition}{Definition}
\newtheorem{example}[theorem]{Example}
\newtheorem*{remark}{Remark}
\newtheorem*{notation}{Notation and terminology}
\newtheorem*{acknowledgements}{Acknowledgements}
\newtheorem*{ethical approval}{Ethical approval}

\newtheorem{openquestions}[theorem]{Open Questions}
\begin{title}{The rank condition and strong rank conditions for Ore extensions}
\end{title}

\begin{document}
\author{
{\sc Karl Lorensen}\\
{\sc Johan \"Oinert}
}

\maketitle

\begin{abstract} Let $R$ be a ring, $\sigma:R\to R$ a ring endomorphism, and $\delta:R\to R$ a $\sigma$-derivation. We establish that the Ore extension $R[x;\sigma,\delta]$ satisfies the rank condition if and only if $R$ does. In addition, we prove analogous results for the right and left strong rank conditions. However, in the right case, the ``if" part requires the hypothesis that $\sigma$ is an automorphism, whereas, in the left case, this assumption is needed for the ``only if" part.  Finally, we provide a new proof of an old result of Susan Montgomery stating that a skew power series ring is  directly (respectively, stably) finite if and only if its coefficient ring is directly (respectively, stably) finite. 

\vspace{10pt}

\vspace{5pt}

\end{abstract}

\section{Introduction}

In 1933, \O{}ystein Ore published a seminal paper \cite{Ore} expounding a general theory of noncommutative polynomials. The structures that he introduced, now referred to as Ore extensions, have come to occupy a position of central importance in the theory of noncommutative rings and have been applied fruitfully in a diverse array of fields, including quantum mechanics and coding theory (see, for instance, \cite{Kassel}, \cite{Boucher}, and \cite{Boucher2}).  

To construct an Ore extension, we begin with an arbitrary ring $R$ and ring endomorphism $\sigma:R\to R$. In addition, we require a \emph{$\sigma$-derivation} $\delta:R\to R$, that is, an additive map fulfilling the condition
\begin{equation*}\delta(rs)=\delta(r)s+\sigma(r)\delta(s)\end{equation*}
for all $r, s\in R$. The \emph{Ore extension} of $R$ with respect to $\sigma$ and $\delta$, denoted $R[x; \sigma, \delta]$, is the ring consisting of all polynomials in the variable $x$, with coefficients in $R$ written on the left, such that the multiplication is governed by the rule
\begin{equation*} xr=\sigma(r)x + \delta(r)\end{equation*}
for all $r\in R$.  For two different proofs of the existence of such a ring structure on the set of left polynomials over a ring, we refer the reader to \cite{Bergman} and \cite[\S 1.1]{Cohn} (also \cite[\S 2.3]{NNR}).  

If $\delta=0$, then $R[x;\sigma,\delta]$ is called the \emph{skew polynomial ring} over $R$ with respect to $\sigma$ and written $R[x;\sigma]$. The skew polynomial rings form an important subclass of the class of Ore extensions, as they are structurally much simpler than other Ore extensions and yet still differ markedly from ordinary polynomial rings. 

The aim of the current paper is to investigate the relationship between the properties of $R[x;\sigma,\delta]$ and those of its coefficient ring $R$. Specifically, we wish to discover properties that are always inherited by $R[x;\sigma,\delta]$ from $R$, as well as ones that, conversely, $R$ inherits from $R[x;\sigma,\delta]$. A very straightforward result of this kind is that, if $\sigma$ is injective, then $R[x;\sigma,\delta]$ is a domain if and only if $R$ is a domain (see \cite[Theorem 2.9(i)]{NNR}). This reference contains a further result in this vein, namely, that, if $\sigma$ is an automorphism, then $R[x;\sigma,\delta]$ is right (respectively, left) Noetherian if and only if $R$ is right (respectively, left) Noetherian \cite[Theorem 2.9(iv)]{NNR}. Also, C. W. Curtis \cite{Curtis} proved in 1952 that, if $\sigma$ is injective, then $R[x;\sigma,\delta]$ is a left Ore domain if $R$ is a left Ore domain (see also  \cite[Proposition 1.1.4]{Cohn}). Moreover, it follows from T. Y. Lam's reasoning for skew polynomial rings in \cite[Theorem 10.28]{Lam} that the converse of Curtis's result holds if $\sigma$ is assumed to be an automorphism. 

The principal ring-theoretic properties that we examine are the rank condition, the left strong rank condition, and the right strong rank condition. Recall that a ring $R$ satisfies the \emph{rank condition} if, for every $n\in \mathbb Z^+$, there is no right $R$-module epimorphism $R^n\to R^{n+1}$, or, equivalently, no left $R$-module epimorphism $R^n\to R^{n+1}$. Another characterization of this property is that, for every $n\in \mathbb Z^+$, there exists a finitely generated right (equivalently, left) $R$-module that cannot be generated by $n$ elements. For this reason, rings that fulfill the rank condition are often described as having \emph{unbounded generating number}. 

A ring $R$ satisfies the \emph{right strong rank condition} (RSRC) if, for every positive integer $n$, there is no right $R$-module monomorphism $R^{n+1}\to R^n$. 
Analogously, a ring $R$ satisfies the \emph{left strong rank condition} (LSRC) if, for every $n\in \mathbb Z^+$, there is no left $R$-module monomorphism $R^{n+1}\to R^n$. These two properties are not equivalent; this is illustrated by Lemmas~1.3 and 1.10 below. 
It is easy to see that, as indicated by their names, both the right and left strong rank conditions imply the rank condition. Also, both strong rank conditions are, for nonzero rings, generalizations of commutativity (see \cite[Corollary 1.38]{Lam}), as well as, obviously, the property of being a division ring. Moreover, RSRC and LSRC are easily seen to generalize the right and left Noetherian properties, respectively (see \cite[Theorem 1.35]{Lam}).  
 
 Introduced by P. M. Cohn \cite{CohnIBN} in 1966, the rank condition has been studied intermittently since then; see, for instance, \cite{Lam}, \cite{CohnSkew}, \cite{Haghany}, \cite{Cohn},  \cite{Abrams}, \cite{LO}, and \cite{LO2}. The two strong rank conditions were, we believe, first formulated by Lam \cite{Lam} in 1999. For the next twenty years, they received barely any further mention in the literature, a notable exception being \cite{Haghany}. Recently, however, they have garnered much more attention because of their connection to the property of amenability for groups and algebras (see \cite{Bartholdi}, \cite{KropLor}, and \cite{LO2}).   

We prove two theorems about the above three properties for Ore extensions;  the first, stated immediately below, pertains to the rank condition.

\begin{theorem} Let $R$ be a ring, $\sigma:R\to R$ a ring endomorphism, and $\delta:R\to R$ a $\sigma$-derivation. Then $R[x;\sigma,\delta]$ satisfies the rank condition if and only if $R$ does.
\end{theorem}

With the strong rank condition, the situation is more complicated: the hypothesis that $\sigma$ is an automorphism is required to ensure the property's inheritability in two instances; also, 
there are discrepancies between the conclusions that can be drawn for the left and right versions, respectively, of the property.  
Our results are summarized in Theorem 0.2 directly below.  

\begin{theorem} Let $R$ be a ring, $\sigma:R\to R$ a ring endomorphism, and $\delta:R\to R$ a $\sigma$-derivation. Then the following four statements hold.
\begin{enumerate*}
\item If $R$ satisfies LSRC, then $R[x;\sigma, \delta]$ satisfies LSRC.
\item If $\sigma$ is an automorphism and $R$ satisfies RSRC, then $R[x;\sigma, \delta]$ satisfies RSRC.
\item If $R[x;\sigma, \delta]$ satisfies RSRC, then $R$ satisfies RSRC.  
\item If $\sigma$ is an automorphism and $R[x;\sigma, \delta]$ satisfies LSRC, then $R$ satisfies LSRC.
\end{enumerate*}
\end{theorem}

To establish both theorems, we employ the notion of a filtration of a ring. With this approach, Theorems 0.1 and 0.2 become corollaries of two more general results about rings with special kinds of filtrations, Propositions 1.8 and 1.9, respectively. 

We point out that a few particular cases and parts of these theorems are already well known and straightforward to prove. For example, the ``only if" part of Theorem 0.1 is trivial since $R$ is a unital subring of $R[x;\sigma,\delta]$. The case of the converse for a skew polynomial ring is also trivial since there is a canonical ring homorphism from such a ring to its coefficient ring. Moreover, the case of Theorem 0.2 for the standard polynomial ring $R[x]$ is an exercise in Lam's book \cite[\S1, Exercise 22]{Lam}. 

As shown in \cite[Theorem~10.21]{Lam}, a domain satisfies RSRC (respectively, LSRC) if and only if it is right
 (respectively, left) Ore. Consequently, parts (i) and (iv) of Theorem~0.2 generalize the results of Curtis and Lam, respectively, about left Ore domains described in the last two sentences of the fourth paragraph of this introduction.

The paper also investigates whether the hypothesis that $\sigma$ is an automorphism in assertions (ii) and (iv) in Theorem 0.2 can perhaps be weakened. A result of Lam \cite{Lam} (Lemma 1.11 below) shows that this condition in Theorem 0.2(ii) cannot be replaced by the assumption that $\sigma$ is an injective endomorphism.  
 However, it remains, in general, unknown whether $\sigma$ being a surjective endomorphism might suffice instead. Regarding this question, we establish the following partial result. 
  
  \begin{corollary} Let $R$ be a domain, $\sigma:R\to R$ a surjective ring endomorphism, and $\delta:R\to R$ a $\sigma$-derivation. If $R$ satisfies RSRC, then so does $R[x;\sigma, \delta]$. 
   \end{corollary}
  
  A similar dilemma arises with regard to statement (iv) of Theorem 0.2. In Example 1.14, we demonstrate that the hypothesis that $\sigma$ is an automorphism in (iv) cannot be replaced by $\sigma$ being merely an injective endomorphism. However, the question of whether the assumption that $\sigma$ is a surjective endomorphism might suffice is left unanswered. 

 The third theorem that we include in the paper is essentially due to M. S. Montgomery, although we provide an alternative proof. The result concerns the ring-theoretic properties of direct finiteness and stable finiteness, which are closely related to the rank condition but much more widely studied. 
Recall that a ring $R$ is \emph{directly finite} (alternatively, \emph{Dedekind finite} or \emph{von Neumann finite}) if, for any $r, s\in R$, $rs=1$ implies that $sr=1$. Moreover, a ring $R$ is \emph{stably finite} (alternatively, \emph{weakly finite}) if the ring of $n\times n$ matrices over $R$ is directly finite for every positive integer $n$. P.~Malcolmson proved \cite{Malcolmson} that a ring fulfills the rank condition if and only if it possesses a nonzero quotient that is stably finite (see also \cite[Theorem~1.29]{Lam}). 
 
 In contrast to the rank condition, neither stable finiteness nor direct finiteness is always inherited by Ore extensions (see Example 1.16). Nevertheless, both properties are always preserved by forming skew polynomial rings and, more generally, skew power series rings. This was shown for ordinary power series rings by Montgomery in \cite[Corollary 2]{Montgomery}, employing the Jacobson radical (see also \cite[\S1, Exercises 13, 15]{Lam}). Moreover, her argument applies equally to skew power series rings. 

 \begin{theorem}[{Montgomery}] Let $R$ be a ring and $\sigma:R\to R$ a ring endomorphism. Then the skew power series ring $R[[x;\sigma]]$ is directly (respectively, stably) finite if and only if $R$ is directly (respectively, stably) finite.
 
 Consequently, the skew polynomial ring $R[x;\sigma]$ is directly (respectively, stably) finite if and only if $R$ is directly (respectively, stably) finite.
 
 \end{theorem}
 
We adopt a different approach than Montgomery to prove Theorem 0.4, obtaining the result instead as a consequence of the following proposition, which may be of independent interest.   
 
 \begin{proposition} Let $R$ be a ring and $\sigma:R\to R$ a ring homomorphism. Then the unit element is the only idempotent in $R[[x;\sigma]]$ that has constant term $1_R$.  
 \end{proposition}
 
As an application of Theorem 0.4, we prove in Corollary 1.18 that the ring of upper triangular $\mathbb N\times \mathbb N$ matrices with entries in a directly (respectively, stably) finite ring is directly (respectively, stably) finite. 

 At the end of the paper, we list five open questions about Ore extensions pertaining to the ring-theoretic properties discussed in the paper.  
 \vspace{10pt}

 \begin{notation}
 \ \
 \vspace{5pt}
 
 {\rm  
The term \emph{ring} will always mean a unital and associative ring. Moreover, ring homomorphisms will always be assumed to preserve unit elements.  However, for our purposes, it will be convenient to distinguish between unital and nonunital  subrings. We will employ the term \emph{subring} of a ring $R$ to refer to any subset of $R$ that forms a ring under the same two operations. If the unit elements of the subring and overring coincide, then we call the subring a \emph{unital subring}; otherwise, it is known as a \emph{nonunital subring}.
\vspace{5pt}

We will always write right module homomorphisms on the left of their arguments and left module homomorphisms on the right.
\vspace{5pt}

The set of \emph{natural numbers}, denoted $\mathbb N$,
 consists of the positive integers and $0$.
\vspace{5pt}

An element $a$ of a ring $R$ is called a  \emph{left zero divisor} if there is a nonzero element $b$ of $R$ such that $ab=0$. 
\vspace{5pt}

Let $R$ be a ring and $M$ a right $R$-module. A finite sequence of elements $a_1,\dots,a_n\in M$  is said to be \emph{right $R$-linearly independent} if, for any $r_1,\dots, r_n\in R$,
\[a_1r_1+\cdots+a_nr_n=0\ \ \Longrightarrow\ \ r_1=r_2=\cdots=r_n=0.\]
Moreover, left linear independence for left modules is defined analogously. 
\vspace{5pt}

Let $R$ be a ring and $n$ a positive integer. We employ $M_n(R)$ and $M_{\mathbb N}(R)$ to denote, respectively, the set of all $n\times n$ and $\mathbb N\times \mathbb N$ matrices over $R$.
Moreover, we write  $UM_{\mathbb N}(R)$ for the ring of \emph{upper triangular} $\mathbb N\times \mathbb N$ matrices over $R$, that is, all matrices $M\in M_{\mathbb N}(R)$ for which $M(i,j)=0$ if $i>j$.
\vspace{5pt}
}
\end{notation}

\begin{acknowledgements} {\rm The authors are indebted to the referee for their careful reading of the manuscript and their astute suggestions, which led to significant improvements to the paper. In addition, we wish to acknowledge the assistance provided by {\it ChatGPT} \cite{ai} in proofreading an earlier version of the proof of Proposition 1.19 and detecting a typo in one of the sums.}
\end{acknowledgements}

\section{Proofs of the results}

We begin by recalling two common alternative characterizations of the rank condition and strong rank conditions, respectively. These amount to merely slight reformulations of the definitions. 

\begin{lemma} For any ring $R$, the following two statements are equivalent.
\begin{enumerate*}
\item $R$ satisfies the rank condition.
\item For every pair of positive integers $m, n$, there exists a right or left $R$-module epimorphism $R^n\to R^m$ only if $n\geq m$. 
\end{enumerate*}
\end{lemma}

\begin{lemma} For any ring $R$, the following two statements are equivalent.
\begin{enumerate*}
\item $R$ satisfies RSRC (respectively, LSRC).
\item For every pair of positive integers $m, n$, there exists a right (respectively, left) $R$-module monomorphism $R^n\to R^m$ only if $n\leq m$. 
\end{enumerate*}
\end{lemma}

The literature contains a plethora of examples of rings that satisfy or fail to satisfy one or both of the strong rank conditions; see, in particular, \cite[\S 1D]{Lam}, \cite{KropLor}, and \cite{LO2}. One very straightforward, elementary family of examples that has, to our knowledge, never been mentioned in print before consists of all rings $UM_{\mathbb N}(R)$, where $R$ is an arbitrary ring satisfying RSRC. In Lemma 1.3, we show that such a ring fulfills RSRC but not LSRC. Part (ii) of the lemma will play a role below in Example 1.14. 

\begin{lemma} The following two statements hold for an arbitrary ring $R$.
\begin{enumerate*}
\item  If $R$ satisfies RSRC, then so does $UM_{\mathbb N}(R)$. 
\item The ring $UM_{\mathbb N}(R)$ fails to satisfy LSRC. 
\end{enumerate*}
\end{lemma}

\begin{proof} Throughout the proof, we will write $I$ for the identity element of $UM_{\mathbb N}(R)$. 

To prove (ii), let $A$ be the matrix in $UM_{\mathbb N}(R)$ such that, for every $i\in \mathbb N$,  $A(i,2i)=1$ and $A(i,j)=0$ if $j\neq 2i$. Also, let $B$ be the upper triangular matrix such that 
$B(i,2i+1)=1$ and $B(i,j)=0$ if $j\neq 2i+1$. Notice that $AA^t=BB^t=I$ and $AB^t=BA^t=0$. It follows, then, easily from these equations that $A$ and $B$ are left $UM_{\mathbb N}(R)$-linearly independent.
Therefore the map $(X,Y)\mapsto XA+YB$ from $\left (UM_{\mathbb N}(R)\right )^2$ to $UM_{\mathbb N}(R)$ is a left $UM_{\mathbb N}(R)$-module monomorphism. Thus $UM_{\mathbb N}(R)$ fails to satisfy LSRC. 

To establish (i), we suppose that $R$ satisfies RSRC. Let $n$ be a positive integer and $\phi:\left (UM_{\mathbb N}(R)\right )^{n+1}\to \left (UM_{\mathbb N}(R)\right )^n$ be a right $UM_{\mathbb N}(R)$-module homomorphism. We will show that $\phi$ must not be injective, which will imply that $UM_{\mathbb N}(R)$ satisfies RSRC. 
Write $U$ for the subset of  $UM_{\mathbb N}(R)$ consisting of all the matrices with every entry other than the $(0,0)$ one equal to zero. Note that we can view $U$ as a right $R$-module via the ring homomorphism $r\mapsto rI$ from $R$ to $UM_{\mathbb N}(R)$. From the latter point of view, we have $U\cong R$. 

Since $U$ is a left ideal in $UM_{\mathbb N}(R)$, we know that $\phi(U^{n+1})\subseteq U^n$. Hence $\phi$ induces a right $R$-module homomorphism $U^{n+1}\to U^n$. Since $R$ satisfies RSRC and $U\cong R$ as a right $R$-module, this means that $\phi$ is not monic.  
\end{proof}

\begin{remark}{\rm As the reader can easily verify, there is a result corresponding to Lemma 1.3 that refers to lower triangular $\mathbb N\times\mathbb N$ matrices, but with the properties RSRC and LSRC interchanged.}
\end{remark}

We now recall some elementary properties of Ore extensions that we require for our proofs of Theorems 0.1 and 0.2. First, we point out that an Ore extension $R[x;\sigma, \delta]$ is plainly freely generated as a left $R$-module by the set $\{1, x, x^2,\dots\}$. Below, we establish the well-known fact that, if $\sigma$ is an automorphism, then the same set also generates $R[x;\sigma,\delta]$ freely as a right $R$-module. 

\begin{lemma} Let $R$ be a ring, $\sigma:R\to R$ a ring automorphism, and $\delta:R\to R$ a $\sigma$-derivation. For each $i\in \mathbb N$, let $U_i$ be the left $R$-submodule of $R[x;\sigma,\delta]$ generated by the set $I_i:=\{1, x,\dots, x^i\}$. Then each $U_i$ is free as a right $R$-module on the set $I_i$. 
\end{lemma}

\begin{proof} The proof is by induction on $i$, the case $i=0$ being trivial. 
Suppose that $i>0$. Notice that 

\begin{equation} x^i\, r\in \sigma^i(r)x^i + U_{i-1}.\end{equation}
Appealing to (1.1) and invoking the bijectivity of $\sigma^i$, we deduce that $U_i/U_{i-1}$ is freely generated as a right $R$-module by the image of $x^i$. Since $U_{i-1}$ is freely generated as a right $R$-module by $I_{i-1}$, this means that $I_i$ generates $U_i$ as a free right $R$-module. The  conclusion of the lemma now follows.
\end{proof}

\begin{corollary} Let $R$ be a ring, $\sigma:R\to R$ a ring automorphism, and $\delta:R\to R$ a $\sigma$-derivation. Then the ring $R[x;\sigma,\delta]$ is freely generated as a right $R$-module by the set $\{1,x,x^2,\dots\}$.
\end{corollary}

In investigating Ore extensions, we will avail ourselves of the concept of a filtration of a ring. 

\begin{definition}{\rm Let $R$ be a ring and $\mathcal{U}=\{U_i:i\in \mathbb N\}$ a set of additive subgroups of $R$. The set $\mathcal{U}$ is a {\it filtration} of $R$ if the following four conditions are satisfied: 
\begin{enumerate*}
\item $U_0\subseteq U_1\subseteq U_2\subseteq\cdots;$
\item $U_iU_j\subseteq U_{i+j}$ for all $i, j\in \mathbb N$;
\item $\bigcup_{i=0}^\infty U_i=R$; 
\item $1\in U_0$.
\end{enumerate*}
If these conditions hold, then we also say that $R$ is {\it filtered} by $\mathcal{U}$.}
\end{definition}

Filtrations of rings are readily seen to enjoy the following properties. 

\begin{lemma} If a ring $R$ is filtered by a set $\{U_i : i\in \mathbb N\}$ of additive subgroups, then $U_0$ is a unital subring of $R$ and, for every $i\in \mathbb N$, $U_i$ is both a right and left $U_0$-submodule.
\end{lemma}

It is easy to see that every Ore extension admits a canonical filtration. 

\begin{lemma} Let $R$ be a ring, $\sigma:R\to R$ a ring endomorphism, and $\delta:R\to R$ a $\sigma$-derivation. For each $i\in \mathbb N$, let $U_i$ be the left $R$-submodule of $R[x;\sigma,\delta]$ generated by the set $\{1, x,\dots, x^i\}$. Then $R[x;\sigma,\delta]$ is filtered by $\{U_0, U_1,\dots\}$.
\end{lemma}

The first main result of the paper, Theorem 0.1, is a special case of the following proposition about filtered rings.  

\begin{proposition}  Let $R$ be a ring filtered by a set $\{U_i : i\in \mathbb N\}$ of additive subgroups. Suppose further that, for all $i\in \mathbb N$, $U_{i+1}/U_i$ is a free left $U_0$-module of rank $1$.
Then $R$ satisfies the rank condition if and only if $U_0$ satisfies the rank condition.
\end{proposition} 

\begin{proof} The ``only if" part follows from the fact that the rank condition is inherited by unital subrings. To prove the ``if" statement, we establish its contrapositive. Suppose that $R$ fails to satisfy the rank condition; that is, there is a left $R$-module epimorphism $\phi:R^n\to R^m$ such that $m>n$.  
 For each $i=1,\dots, m$, define ${\bf e}_i$ to be the element of $R^m$ that has $1$ in the $i$th position and $0$ everywhere else. Let ${\bf a}_1,\dots, {\bf a}_m\in R^n$ such that ${\bf a}_i\phi={\bf e}_i$ for $i=1,\dots, m$. Take $k$ to be a natural number such that ${\bf a}_i\in U^n_k$ for $i=1,\dots,m$. Select $l$ to be a natural number larger than $\frac{nk}{m-n}-1$.  
Recognizing that $U_l$ is a left $U_0$-module direct summand in $R$, we let $\pi^m_l:R^m\to U_l^m$ be the projection $U_0$-module epimorphism. 

We claim that $U^m_l\subseteq (U^n_{k+l})\phi$. To see this, let $r\in U_l$ and notice that
\[r{\bf e}_i=r({\bf a}_i\phi)=(r{\bf a}_i)\phi\in (U^n_{k+l})\phi\]
for every $i=1,\dots, m$. Since the set $\{r{\bf e}_i:r\in U_l,\ i=1,\dots,m\}$ generates $U^m_l$ as an additive abelian group, this observation verifies the claim. 

Now let $\phi_{k+l}:U^n_{k+l}\to R^m$ be the restriction of $\phi$ to $U^n_{k+l}$. In view of the above claim, the composition $\phi_{k+l}\pi^m_l:U^n_{k+l}\to U^m_l$ must be a left $U_0$-module epimorphism. This gives rise to a left $U_0$-module epimorphism $U_0^{n(k+l+1)}\to U_0^{m(l+1)}$. However, our choice of $l$ renders $n(k+l+1)<m(l+1)$, which implies that $U_0$ fails to fulfill the rank condition. 
\end{proof}

With the aid of Lemma 1.4,  our second main result, Theorem 0.2, can be deduced immediately from our next proposition. 

\begin{proposition}  Let $R$ be a ring filtered by a set $\{U_i:i\in \mathbb N\}$ of additive subgroups. Suppose further that, for all $i\in \mathbb N$, $U_i$ is a free left (respectively, right) $U_0$-module of rank $i+1$.
Then the following two statements hold.
\begin{enumerate*}
\item If $U_0$ satisfies LSRC (respectively, RSRC), then $R$ satisfies LSRC (respectively, RSRC).
\item If $R$ satisfies RSRC (respectively, LSRC) then $U_0$ satisfies RSRC (respectively, LSRC). 
\end{enumerate*}
\end{proposition} 

\begin{proof}  We will just prove the result for the $U_i$ being free left modules; a dual argument can be used for the other version. First we dispose of statement (ii).  Assume that $R$ satisfies RSRC.  Then $R$ is free as a left $U_0$-module, which means that $U_0$ must satisfy RSRC by \cite[Proposition 2.14]{LO2}.  

To prove (i), we establish the contrapositive. Suppose that $R$ fails to satisfy LSRC. This means that there are $m,n\in \mathbb Z^+$ with $m>n$ and a left $R$-module monomorphism $\phi:R^m\to R^n$. Let $k$ be the largest natural number such that $U_k$ contains all the entries in the matrix of $\phi$ with respect to the standard bases for $R^m$ and $R^n$. Pick $l$ to be a natural number larger than $\frac{nk}{m-n}-1$. The map $\phi$ induces a left $U_0$-module monomorphism $U^m_l\to U^n_{k+l}$, yielding a left $U_0$-module monomorphism $U_0^{m(l+1)}\to U_0^{n(k+l+1)}$.
But our choice of $l$ ensures that $m(l+1)>n(k+l+1)$, which implies that $U_0$ fails to satisfy LSRC.
\end{proof}

The hypothesis that $\sigma$ is an automorphism cannot be dropped from (ii) in Theorem 0.2. This is shown by Lemma 1.10 below, which is established in the second paragraph of \S 10C in Lam's book \cite{Lam}. 

\begin{lemma}[{Lam}] Let $D$ be a division ring and $\sigma:D\to D$ a ring homomorphism. If $\sigma$ is not surjective, then $D[x;\sigma]$ does not satisfy RSRC. 
\end{lemma}

\begin{remark}{\rm For an example of a division ring $D$ and a nonsurjective ring endomorphism $\sigma:D\to D$, take $D$ to be the Laurent series ring $K[[x]][x^{-1}]$ over a field $K$, and define $\sigma:D\to D$ to be the $K$-algebra homomorphism $D\to D$ such that $\sigma(x)=x^2$.}\end{remark}

The map $\sigma$ in Lemma 1.10 must be injective. Hence the lemma shows that the hypothesis in Theorem 0.2(ii) cannot be weakened to $\sigma$ being injective. However, we do not know whether statement (ii) remains true if the hypothesis is weakened to $\sigma$ being surjective. This is question (1) on the list of Open Questions 1.20 at the end of the paper. The only contribution that we are able to make to answering this question is Corollary 0.3 from the introduction, which furnishes a positive result in the case that $R$ is a domain. To deduce Corollary 0.3, we require the following lemma and proposition. 

\begin{lemma} Let $R$ be a ring, $\sigma:R\to R$ a ring endomorphism, and $\delta:R\to R$ a $\sigma$-derivation. If $a\in {\rm Ker}\ \sigma$, then, within the Ore extension $R[x;\sigma, \delta]$, we have $x^ia^i\in R$ for all $i\in \mathbb N$. 
\end{lemma}

\begin{proof} We argue by induction on $i$, the case $i=0$ being trivial.  Assume $i>0$ and $r:=x^{i-1 }a^{i-1}\in R$. Then $x^i a^i=xra=\sigma(ra)x+\delta(ra)=\delta(ra)\in R$. 
\end{proof}

\begin{proposition} Let $R$ be a ring that satisfies RSRC. Furthermore, let $\sigma:R\to R$ be a ring endomorphism and $\delta:R\to R$ a $\sigma$-derivation. If ${\rm Ker}\ \sigma$ fails to consist entirely of left zero divisors, then $R[x; \sigma, \delta]$ must satisfy RSRC. 
\end{proposition}

\begin{proof}  Assume that there is an element $a$ of ${\rm Ker}\ \sigma$ that is not a left zero divisor. Put $S:=R[x;\sigma,\delta]$, and let $\phi:S^{n+1}\to S^n$ be a right $S$-module homomorphism.  Invoking the filtration  $\{U_i : i\in \mathbb N\}$ of $R[x; \sigma, \delta]$ described in Lemma 1.7, take $k$ to be a natural number such that $U_k$ contains all of the entries of the matrix of $\phi$ with respect to the standard bases of $S^{n+1}$ and $S^n$. Define the right $R$-module homomorphism $\psi:R^{n+1}\to S^n$ by $\psi{\bf r}:=\phi( a^k{\bf r})$ for all ${\bf r}\in R^{n+1}$. By Lemma 1.11, we have $x^ia^k\in R$ for $0\leq i\leq k$. This implies that ${\rm Im}\ \psi$ is contained in $R^n$. Hence the fact that $R$ satisfies RSRC implies that there is a nonzero element ${\bf b}$ in the kernel of $\psi$. But this then means that $a^k{\bf b}$ is a nonzero element of the kernel of $\phi$. Therefore $S$ fulfills RSRC. 
\end{proof}

\begin{proof}[Proof of Corollary 0.3] Assume that $R$ satisfies RSRC. If $\sigma$ is injective, then $R[x; \sigma, \delta]$ satisfies RSRC by Theorem 0.2(ii). Moreover, if $\sigma$ is not injective, then Proposition 1.12 implies that $R[x; \sigma, \delta]$ must satisfy RSRC. 
\end{proof}

The following example illustrates that the hypothesis that $\sigma$ is an automorphism can also not be eliminated from Theorem 0.2(iv). 

\begin{example}{\rm Let $S$ be the ring with generators $u, v, x$ subject only to the relations $xu=xv=0$; in other words, $S:=\mathbb Z\langle u,v,x\rangle /(xu, xv)$. Note that the polynomial ring $P:=\mathbb Z[x]$ is a unital subring of $S$. We claim that $S$ satisfies LSRC. To show this, let $\phi:S^m\to S^n$ be a left $S$-module homomorphism with $m>n$. Our goal is to prove that $\phi$ cannot be monic, which will establish the claim. Define the left $P$-module homomorphism $\psi:P^m\to S^n$ by ${\bf p}\psi:=({\bf p}x)\phi$ for all ${\bf p}\in P^m$.  Since $Px$ is a right ideal in $S$, we have $\left (P^mx\right )\phi\subseteq P^nx$, implying that ${\rm Im}\ \psi\subseteq P^n$. Since $P$ satisfies LSRC (by, for example, Theorem 0.2(i)), we conclude that there is a nonzero element ${\bf p}\in P^m$ such that ${\bf p}\psi=0$. This makes ${\bf p}x$ a nonzero element of ${\rm Ker}\ \phi$, so that $\phi$ is not monic. 

Finally, we show that $S$ is an Ore extension of a ring that fails to satisfy LSRC. Notice that  $S\cong R[x;\sigma]$, 
where $R$ is the free ring with generators $u, v$ and $\sigma:R\to R$ is the endomorphism that maps each element of $R$ to the constant term of its normal form. Moreover, it is well known that $R$ fails to satisfy LSRC; see, for instance, \cite[Example 1.31]{Lam}. }
\end{example}

The endomorphism $\sigma$ in Example~1.13 is neither injective nor surjective. This suggests the question of whether either of these properties alone might be enough for part (iv) of Theorem~0.2. In Example~1.14, we show that injectivity on its own does not suffice. 
Whether surjectivity might be sufficient remains unanswered and appears as question (2) in Open Questions~1.20. 

\begin{example}{\rm  Let $R$ be a ring that fulfills LSRC, and write $S:={\rm UM}_{\mathbb N}(R)$. Furthermore, let $\sigma:S\to S$ be the ring monomorphism defined by 
\[\sigma(A)(i,j):=\begin{cases} A(0,0) & \text{if}\ (i,j)=(0,0)\\
0 & \text{\rm if}\ i=0\ \text{\rm and}\ j>0\\
 0 & \text{\rm if}\ j=0\ \text{\rm and}\ i>0\\
 A(i-1,j-1) & \text{\rm if}\ i,j>0\end{cases}\]
 for every matrix $A\in S$.
 
 By Lemma 1.3, $S$ does not satisfy LSRC. However, we maintain that the skew polynomial ring $S[x;\sigma]$ fulfills LSRC. 
To argue this, let $E$ be the element of $S$ such that $E(0,0)=1$ and $E(i,j)=0$ otherwise. We will utilize the nonunital subring $U$ of $S[x;\sigma]$ consisting of all elements of the form $\sum_{i=0}^k (r_iE)x^i$ with $r_i\in R$. Notice that the unit element of $U$ is $E$, and that there is a ring isomorphism from $U$ to the polynomial ring $R[t]$ mapping $(rE)x$ to $rt$ for every $r\in R$. Hence $U$ must satisfy LSRC. 
 
 Let $m>n$ be positive integers and $\phi:\left (S[x;\sigma]\right )^m\to \left (S[x; \sigma]\right )^n$ be a left $S[x;\sigma]$-module homomorphism. Define the left $U$-module homomorphism $\psi:U^m\to \left (S[x; \sigma]\right )^n$ by ${\bf u}\psi:=\left ({\bf u}x\right )\phi$ for all ${\bf u}\in U^m$. Observing that $Ux$ is a right ideal of $S[x;\sigma]$, we deduce that  $(U^mx)\phi\subseteq U^nx$. Thus we have ${\rm Im}\ \psi\subseteq U^n$. 
Because $U$ satisfies LSRC, there is a nonzero ${\bf u}\in {\rm Ker}\ \psi$, making ${\bf u}x$ a nonzero element of ${\rm Ker}\ \phi$. Therefore $S[x;\sigma]$ satisfies LSRC.}
 \end{example}
 
The ring $S$ in Example 1.14 is obviously not a domain. Indeed, it remains a mystery whether it may be possible to construct an example of a skew polynomial ring, or an Ore extension, that is a domain satisfying LSRC (that is, a left Ore domain) but whose coefficient ring violates the condition. This question is also included in Open Questions~1.20 (see question (3)). 

Below in Corollary 1.15, we present an application of Theorems 0.1 and 0.2 to the study of Weyl rings, an important family of noncommutative rings that arise in quantum mechanics and other contexts.  

\begin{definition}{\rm Let $R$ be a ring. The \emph{first Weyl ring} over $R$, denoted $W_1(R)$, is the $R$-ring with the presentation

\[W_1(R):=\langle x,y\ :\ xy - yx =1;\ \ rx=xr\ \mbox{and}\ ry=yr\ \mbox{for all}\ r\in R\rangle.\]

For any integer $n>1$, the \emph{$n$th Weyl ring} $W_n(R)$ over $R$ is defined by $W_n(R):=W_1\left (W_{n-1}\left (R\right )\right )$.}
\end{definition}

\begin{corollary} Let $\mathcal{P}$ denote one of the three properties: the rank condition, RSRC, or LSRC. Then a ring $R$ satisfies $\mathcal{P}$ if and only if $W_n(R)$ satisfies $\mathcal{P}$ for every $n\in \mathbb Z^+$.
\end{corollary}

\begin{proof} The general statement will follow from the case where $n=1$. For this case, we will use the fact that $W_1(R)=R[y][x; I_{R[y]}, \delta]$, where $I_{R[y]}:R[y]\to R[y]$ is the identity map and $\delta:R[y]\to R[y]$ is the standard differentiation map. The conclusion can then be deduced by applying Theorems~0.1 and 0.2. 
\end{proof}

\begin{remark} {\rm We point out that Corollary 1.15 is not really a new result. The case of the corollary for the rank condition is proved in \cite[Corollary 3.9]{LO} by invoking the main result in that paper as well as the standard $\mathbb Z$-grading on $W_1(R)$. Moreover, the cases for RSRC and LSRC can also be easily deduced by employing the $\mathbb Z$-grading and appealing to \cite[Proposition A]{KropLor}.}
\end{remark}

Next we briefly address the properties of direct and stable finiteness. We present first an example that shows that these two properties are not always preserved by Ore extensions.
\begin{example}{\rm We make use of a particular construction of an Ore extension that we learned from S. Deprez \cite{Deprez}. For this, we define the ring endomorphism $\sigma: \mathbb Z[y]\to \mathbb Z[y]$ by $\sigma(\sum_{i=0}^n a_iy^i):=a_0$. Also, define the map $\delta:\mathbb Z[y]\to \mathbb Z[y]$ by $\delta(
\sum_{i=0}^n a_iy^i):=\sum_{i=1}^n a_iy^{i-1}$. It is easy to check that $\delta$ is a $\sigma$-derivation. 

Now consider the Ore extension $S:=\mathbb Z[y][x;\sigma,\delta]$. We claim that $S$ is not directly finite, even though $\mathbb Z[y]$ is stably finite. To verify this, notice that
\[xy=\sigma(y)x+\delta(y)=1.\]
Since $yx\neq 1$, this means that $S$ fails to be directly finite.}
\end{example}

Although they are not preserved by Ore extensions in general, both direct and stable finiteness are preserved by forming skew polynomial rings and, more generally, skew power series rings (Theorem 0.4). We will deduce this from Proposition 0.5 from the introduction, proved directly below. 

\begin{proof}[Proof of Proposition 0.5] For any element $p\in R[[x;\sigma]]$ and $i\in \mathbb N$, we let $\pi_i(p)$ denote the coefficient of $x^i$ in the canonical representation of $p$. Let $e\in R[[x; \sigma]]$ such that $e^2=e$ and $\pi_0(e)=1$. For each $i\in \mathbb N$, set $r_i:=\pi_i(e)$. 
Note that
\begin{displaymath}
	e=e^2= \left( 1 + \sum_{i\geq 1} r_i x^i \right)^2
	= 1 + \sum_{i\geq 1} 2r_i x^i + \sum_{i,j \geq 1} r_i \sigma^i(r_j) x^{i+j}.
\end{displaymath}

Seeking a contradiction, suppose that $S:=\{i \geq 1 \mid r_i \neq 0\}$ is non-empty.
Let $k$ be the smallest number in $S$.
We then get that
$r_k=\pi_k(e)=\pi_k(e^2)=2 r_k$.
Hence $r_k=0$, which yields a contradiction.
This shows that $e=r_0=1$.
\end{proof}

For the proof of Theorem 0.4, we also require the following lemma; its proof is an easy exercise.

\begin{lemma} Let $R$ be a ring and $\sigma:R\to R$ a ring homomorphism. Let $n$ be a positive integer, and let $\sigma^\ast:M_n(R)\to M_n(R)$ be the ring homomorphism induced entrywise by $\sigma$. Then 
\[M_n\left (R[[x;\sigma]]\right )\cong M_n(R)[[x;\sigma^\ast]]\]
as rings.
\end{lemma}

\begin{proof}[Proof of Theorem 0.4] We prove the statement for direct finiteness; the one for stable finiteness will then follow from Lemma 1.17. The ``only if" statement follows from the fact that subrings of directly finite rings are directly finite. To prove the opposite implication, assume
that $R$ is directly finite. Let $p, q\in R[[x;\sigma]]$ such that $pq$=1. Then $qp$ is an idempotent with constant term $1$. Therefore $qp=1$ by Proposition 0.5.
\end{proof}

We can apply Theorem 0.4 to prove the following corollary.  

\begin{corollary} Let $R$ be a ring. Then $UM_{\mathbb N}(R)$ is directly (respectively, stably) finite if and only if $R$ is directly (respectively, stably) finite. 
\end{corollary}

To see how Corollary 1.18 follows from Theorem 0.4, we need the isomorphism described below. This isomorphism is straightforward to see and undoubtedly well known to experts. However, we are not aware of an explicit reference, so we provide a proof. 
 
\begin{proposition} Let $R$ be a ring and set $P:=\prod_{\mathbb N} R$. Define the ring endomorphism  $\sigma:P\to P$ by $\sigma\left (\left (r_i\right )_{i\in \mathbb N}\right ):=\left (r_{i+1}\right )_{i\in \mathbb N}$.
Furthermore, define 
 a map  $\theta: UM_{\mathbb N}(R)\to P[[x; \sigma]]$ by, for each $M\in  UM_\mathbb N(R)$, letting 
\[\theta(M):=\sum_{j=0}^\infty \left (M\left (i,i+j\right )\right )_{i\in \mathbb N}x^j.\] 
Then $\theta$ is a ring isomorphism. 
\end{proposition}

\begin{proof} The map $\theta$ is  plainly additive and bijective.  To see that it is also multiplicative, let $M, N\in UM_{\mathbb N}(R)$. Then

\begin{align*}\theta(M)\theta(N) = & \sum_{j=0}^\infty\left ( \sum_{k=0}^{j}\left (M\left (i,i+k\right )\right )_{i\in \mathbb N}\cdot \sigma^k\left (\left (N\left (i,i+j-k\right ) \right )_{i\in \mathbb N}\right )\right )x^j\\
= & \sum_{j=0}^\infty \left (\sum_{k=0}^{j}\left (M\left (i,i+k)\cdot N(i+k,i+j\right )\right )_{i\in \mathbb N}\right )x^j \\
= & \sum_{j=0}^\infty \left (\sum_{k=0}^{j}M\left (i,i+k)\cdot N(i+k,i+j\right )\right )_{i\in \mathbb N}x^j\\
= & \sum_{j=0}^\infty \left (\left (MN\right )(i,i+j)\right )_{i\in \mathbb N}x^j\\
= & \theta(MN). 
\end{align*} 
\end{proof}

\begin{proof}[Proof of Corollary 1.18] We prove the statement for direct finiteness; the one for stable finiteness follows by an analogous argument. The ``only if" direction follows from the fact that subrings of directly finite rings are directly finite. For the opposite implication, assume that $R$ is directly finite. Then $\prod_{\mathbb N} R$ is plainly directly finite. Invoking Theorem~0.4 and Proposition~1.19, we conclude that $UM_{\mathbb N}(R)$ is directly finite. 
\end{proof}

\begin{remark} {\rm Corollary 1.18 can alternatively be proved by considering the Jacobson radical of $UM_{\mathbb N}(R)$, which can be easily seen to consist of all the matrices whose diagonal entries lie in the Jacobson radical of $R$. 
An appeal to \cite[Lemma 2]{Montgomery}, then, completes the argument.
}
\end{remark}

\begin{remark}{\rm Clearly, there is also a result corresponding to Corollary~1.18 for lower triangular matrices.} \end{remark}

As promised, we conclude the paper with a list of questions that still remain open about RSRC, LSRC, direct finiteness, and stable finiteness. 

\begin{openquestions}{\rm  Let $R$ be a ring, $\sigma:R\to R$ a ring endomorphism, and $\delta:R\to R$ a $\sigma$-derivation. 
\begin{enumeratenum}
\item If $R$ satisfies RSRC and $\sigma$ is surjective, must $R[x; \sigma, \delta]$ satisfy RSRC?
\item  If $R[x; \sigma, \delta]$ satisfies LSRC and $\sigma$ is surjective, must $R$ satisfy LSRC?
\item If $R$ is a domain, $\sigma$ is injective, and $R[x; \sigma, \delta]$ satisfies LSRC, must $R$ also satisfy LSRC?
\item If $R$ is directly finite and $\sigma$ is an automorphism, must $R[x; \sigma, \delta]$ be directly finite?
\item  If $R$ is stably finite and $\sigma$ is an automorphism, must $R[x; \sigma, \delta]$ be stably finite?
\end{enumeratenum}}
\end{openquestions}

\vspace{10pt}

\noindent {\sc Karl Lorensen}\\
Department of Mathematics and Statistics\\
Pennsylvania State University, Altoona College\\
Altoona, PA 16601, USA\\
E-mail: {\tt kql3@psu.edu}
\vspace{10pt}

\noindent {\sc Johan \"Oinert}\\
Department of Mathematics and Natural Sciences\\
Blekinge Institute of Technology\\
SE-37179 Karlskrona, Sweden\\
E-mail: {\tt johan.oinert@bth.se}
\vspace{3pt}

\noindent {\it and}
\vspace{3pt}

\noindent Department of Engineering\\
University of Sk\"{o}vde\\
SE-54128 Sk\"{o}vde, Sweden


\begin{thebibliography}{21}

\bibitem[{\bf 1}]{Abrams}{\sc G. Abrams, T. G. Nam,} and {\sc N. T. Phuc.} Leavitt path algebras having unbounded generating number. {\it J.  Pure Appl. Algebra} {\bf 221} (2017), 1322-1343. 

\bibitem[{\bf 2}]{Bartholdi} {\sc L. Bartholdi.}  Amenability of groups is characterized by Myhill's theorem (with an appendix by D. Kielak). {\it J.  Eur. Math. Soc.} {\bf 21} (2019), 3191-3197. 

\bibitem[{\bf 3}]{Bergman}{\sc G. M. Bergman.} The diamond lemma for ring theory. {\it Adv. Math.} {\bf 29} (1978), 178-218.

\bibitem[{\bf 4}]{Boucher}{\sc D. Boucher, W. Geiselmann,} and {\sc F. Ulmer.} Skew-cyclic codes. {\it Appl. Algebra Engrg. Comm. Comput.} {\bf18} (2007), 379--389.

\bibitem[{\bf 5}]{Boucher2}{\sc D. Boucher} and {\sc F. Ulmer.}
Linear codes using skew polynomials with automorphisms and derivations.
{\it Des. Codes Cryptogr.} {\bf 70} (2014), 405–431.

\bibitem[{\bf 6}]{CohnIBN}{\sc P. M. Cohn.} Some remarks on the invariant basis property. {\it Topology} {\bf 5} (1966), 215-228. 

\bibitem[{\bf 7}]{CohnSkew}{\sc P. M. Cohn.} {\it Skew Fields: Theory of General Division Rings.} Encyclopedia of Mathematics and its Applications {\bf 57}. Cambridge, 1995. 

\bibitem[{\bf 8}]{Cohn}{\sc P. M. Cohn}. {\it Free Ideal Rings and Localization in General Rings}. Cambridge, 2006. 

\bibitem[{\bf 9}]{Curtis}{\sc C. W. Curtis}. A note on noncommutative polynomial rings. {\it Proc. Amer. Math. Soc.} {\bf 3} (1952), 965-969.

\bibitem[{\bf 10}]{Deprez}{\sc S. Deprez.} Simple Ore extensions. {\it URL (version 2011-11-17): https://mathoverflow.net/q/81183.}

\bibitem[{\bf 11}]{Haghany}{\sc A. Haghany} and {\sc K. Varadarajan.} IBN and related properties for rings. {\it Acta. Math. Hungar.} {\bf 94} (2002), 965-969.

\bibitem[{\bf 12}]{Kassel}{\sc C. Kassel}. {\it Quantum Groups}. Springer, 1995. 

\bibitem[{\bf 13}]{KropLor}{\sc P. H. Kropholler} and {\sc K. Lorensen}. The strong rank condition for group-graded rings. {\it J. Algebra} {\bf 539} (2019), 251-261.

\bibitem[{\bf 14}]{Lam}{\sc T. Y. Lam}. {\it Lectures on Modules and Rings}. Springer, 1999. 

\bibitem[{\bf 15}]{LO}{\sc K. Lorensen} and {\sc J. \"Oinert}. Generating numbers of rings graded by amenable and supramenable groups. {\it J. London Math. Soc} {\bf 109} (2024), no. 1, Paper No. e12826, 34 pp. 

\bibitem[{\bf 16}]{LO2}{\sc K. Lorensen} and {\sc J. \"Oinert}. Rank conditions and amenability for rings associated to graphs. arXiv:2404.07093.

\bibitem[{\bf 17}]{Malcolmson}{\sc P. Malcolmson.} On making rings weakly finite. {\it Proc. Amer. Math. Soc.} {\bf 80} (1980), 215-218.

\bibitem[{\bf 18}]{NNR}{\sc J. C. McConnell} and {\sc J. C. Robson.} {\it Noncommutative Noetherian Rings,} Graduate Studies in Mathematics {\bf 30}. American Mathematical Society, 1987. 

\bibitem[{\bf 19}]{Montgomery}{\sc M. S. Montgomery.} Von Neumann finiteness of tensor products of algebras. {\it Comm. Algebra} {\bf 11} (1983), 595-610. 

\bibitem[{\bf 20}]{ai}{\sc OpenAI.} {\it ChatGPT} (Version from February 19, 2026). https://chat.openai.com/. 

\bibitem[{\bf 21}]{Ore}{\sc \O{}. Ore}. Theory of non-commutative polynomials. {\it Ann. Math.} {\bf 34} (1933), 480-508.
\end{thebibliography}
\end{document}